\documentclass[psamsfonts]{amsart}

\usepackage{amssymb,amsfonts,amscd}

\usepackage[colorlinks,linktocpage]{hyperref}
\hypersetup{linkcolor=[rgb]{0,0,0.715}}
\usepackage{mathtools}
\usepackage{ifsym}
\usepackage[usenames,dvipsnames]{xcolor}

\newtheorem{thm}{Theorem}[section]
\newtheorem{cor}[thm]{Corollary}
\newtheorem{prop}[thm]{Proposition}
\newtheorem{lem}[thm]{Lemma}

\theoremstyle{definition}

\theoremstyle{remark}
\newtheorem{rem}[thm]{Remark}

\newcommand{\vol}{\mbox{vol}}

\DeclareMathOperator{\tr}{tr}

\DeclareMathOperator{\sn}{sn}
\DeclareMathOperator{\cs}{cs}
\DeclareMathOperator{\Ric}{Ric}

\def \lims {\lim\limits}

\numberwithin{equation}{section}

\title[]{Fenchel-Willmore-Chen  Inequality under Lower Bounds on weighted intermediate Ricci curvature}
\author[]{Jihye Lee}
\address[Jihye Lee]{School of Mathematical and Physical Sciences, Macquarie University, Sydney, NSW, Australia}
\email{jihye.lee@mq.edu.au}
\thanks{}

\author[]{Guofang Wei}
\address[Guofang Wei]{Department of Mathematics, University of California,  Santa Barbara, CA, USA.}
\email{wei@math.ucsb.edu}
\thanks{G. Wei is partially supported by NSF DMS  2403557.}
\thanks{J. Lee is supported by Discovery Projects Grant DP250103808 of the Australian Research Council.}

\begin{document}
	
	\begin{abstract}
	 	We establish a Fenchel–Willmore–Chen inequality for smooth metric
measure spaces with a lower bound on the 1-Bakry–Emery $k$-Ricci curvature. This
extends previous results and yields a unified proof. In addition, a Sobolev and isoperimetric inequality for nonnegative 1-weighted Ricci curvature is proven. 
	\end{abstract}
	
	\maketitle

    The Fenchel–Willmore–Chen inequality establishes the sharp scale-invariant lower bound for the total mean curvature of a compact submanifold $M^n$ immersed in $\mathbb R^{n+m}.$  It connects to several other geometric inequalities and plays an important role in the study of submanifolds. The classical Fenchel theorem \cite{Fenchel1929} says for a closed space curve with curvature
$\kappa$, $$\int_{\gamma} \kappa(s) \, ds \geq 2\pi$$
with equality only for convex planar curves. 
 
 For a closed compact surface $\Sigma$ in $\mathbb{R}^3$ with mean curvature $H= \tfrac{\lambda_1 + \lambda_2}{2}$, Willmore \cite{Willmore68} proved that $$\int_{\Sigma} H^2 \, dA \geq 4\pi$$
 with equality if and only if $\Sigma$ is a round sphere. The left-hand side is referred to as the Willmore energy which is conformal invariant, plays an important role in elasticity of membranes, string/surface models in physics. 

In 1971 Bang-Yen Chen \cite{Chen71} generalized it to submanifolds of any dimension and codimension.
For a closed compact  submanifold $\Sigma^n$ immersed in $\mathbb R^{n+m}$, let $\vec{H}= \tfrac 1n \mbox{tr} II$ be the mean curvature vector:
\begin{equation}
    \int_{\Sigma} |\vec{H}|^n \, dV \geq \vol ({\mathbb S}^n)  \label{Chen inequality} 
    \end{equation}
 with equality (for $n>1$) if and only if $\Sigma^n$ is an umbilical hypersphere in $\mathbb R^{n+1}$.

Recently, there have been many extensions to much more general ambient spaces with lower-curvature bounds. Agostiniani, Fogagnolo, and Mazzieri \cite{AFM2020} extended to compact hypersurfaces in manifolds with nonnegative Ricci curvature using a monotonicity
formula arising from potential theory. As pointed out in \cite{brendle2026} the result also follows from Heintze-Karcher's \cite{Heintze-Karcher} comparison estimate. See also the proof by Wang \cite{XW_willmore}, Brendle \cite[Theorem 6.1]{brendle2026} along this line. Ji-Kwong \cite{JK2025}, Pan-Yi \cite{Pan-Yi2026} extended to higher codimension for nonnegative $k$-Ricci curvature. Other extensions include compact hypersurfaces in negative Ricci curvature lower bound \cite{jin-yin}, asymptotically nonnegative Ricci curvature \cite{rudnik2023}, weighted curvature case \cite{wuwu2024, rudnik2025}, notably the substatic case \cite{borghini2023comparison}. The first named author \cite{lee2025willmore} also extended to  asymptotically negative and integral Ricci curvature case.

In this paper we generalize Fenchel–Willmore–Chen inequality to smooth metric
measure spaces with a lower bound on the 1-Bakry–Emery $k$-Ricci curvature, recovering almost of all previous work. 

    Recall for a smooth metric measure space $(M^n, g, e^{-f}d\vol_g)$, the $N$-Bakry-Emery Ricci curvature, $\operatorname{Ric}_f^N$,  $N \in (-\infty, +\infty]$,  is 
$$\operatorname{Ric}_f^N
:=
\operatorname{Ric}
+\operatorname{Hess}f
-\frac{1}{N-n}df\otimes df.$$
When $N=n$, one assumes that $f$ is constant and  $\Ric_f^n:= \Ric$.  

As the quantity is monotone in $N$, we have 
%goes from $n \ra +\infty$ and from $-\infty \ra n$ the quantity increases. Namely
\[ \Ric_f^1 \ge \Ric_f^0 \ge \Ric_f^\infty \ge \operatorname{Ric}_f^N \ \mbox{for all} \ N >n.
\]
Weighted Ricci curvature occurs naturally as Ricci curvature of collapsed spaces (when $N > n$).  $\Ric_f^\infty = \lambda g$ is the Ricci soliton equation, and $\Ric_f^1 = 0$ is the static vacuum Einstein equation. 

For any $x \in M^n$, $v \in T_xM$ unit vector and a $k$-dimensional plane $P$ in $T_x M$ with $v \perp P$, the $k$-Ricci curvature, 
$\operatorname{Ric}_k$,  is 
$$\operatorname{Ric}_k (v, P) = \sum_{i=1}^k K(v, E_i),$$
where
$\{E_i\}_{i=1}^k$ is an orthonormal basis of $P$ and $K(v, E_i)$ is the sectional curvature.

Note that $\Ric_1 =K$ and $\Ric_{n-1} = \Ric$. Hence, it bridges between  sectional  and Ricci curvatures. 

Combining the above together, we consider the weighted $\ell$ -Ricci curvature. 
$$
\operatorname{Ric}^{1}_{\ell,f}(v, P)
\coloneqq
\operatorname{Ric}_{\ell}(v,P)
+
\frac{\ell}{n-1}\operatorname{Hess} f(v,v)
+
\frac{\ell}{(n-1)^2}df\otimes df (v, v)
$$

When $\ell =n-1$, $\operatorname{Ric}^{1}_{\ell,f}(v, P) = \Ric_f^1$. 
When $\ell = 1$, this agrees with the weighted sectional curvature appearing in Wylie \cite{Wylie2015} (See also \cite[Page 3]{wylie2016geometry}).
We also note that this curvature condition satisfies the same monotonicity
property as the usual intermediate Ricci curvature. Namely, if
$1 \leq k\leq \ell \leq n-1$,
then
$$
\operatorname{Ric}^{1}_{k,f}\geq 0
\quad\Longrightarrow\quad
\operatorname{Ric}^{1}_{\ell,f}\geq 0.
$$

    \begin{thm} \label{main theorem}
Let $(M^{n+m}, g, e^{-f}dvol)$ be a complete noncompact smooth metric measure space,  $\Sigma^n$ a closed immersed submanifold of $M^{n+m}$. 
Assume that
\begin{equation}\label{Curvature-assumption}
    \operatorname{Ric}^{1}_{n,f} (v, P)
\geq
n K e^{-\frac{4f}{n+m-1}} \ \ \ \mbox{for} \ K \le 0. 
\end{equation}
Then
\[
 \int_\Sigma e^{f(x)}\!\int_{S^\perp_x \Sigma} \!\left( \! \sqrt{-K} e^{-\frac{2f(x)}{n+m-1}} + \!\left\langle \! -\vec H_f(x), y \right\rangle\right)_+^n\,dy d\sigma(x) 
 \ge \operatorname{RV}_{\mu,K}(\Sigma), 
\]
where $\vec H_f(x) = \vec H(x) + \frac{(\nabla f(x))^\perp}{n+m-1}$,
$\operatorname{RV}_{\mu,K}(\Sigma)$ is the weighted relative volume ratio with respect to a reparameterized distance, see \eqref{Volume ratio}. 
\end{thm}
For the equality case, see Proposition~\ref{prop:equality-Pulledback}. 
%See Subsection~\ref{subsection equality} for the discussion of 

\begin{rem}
   In most previous results involving the $k$-Ric curvature, the assumption for $k$ depends on both the dimension and the codimension of the submanifold. Our assumption depends only on the dimension because we use a decomposition for the Jacobian of the normal exponential map in \cite[Lemma 3.2]{Pan-Yi2026}.    
\end{rem}
\begin{rem} There is no restriction on $f$ here. Instead, we use a reparametrized distance introduced in \cite{wylie2016geometry}. 
   When $f$  is bounded, the lower curvature bound, the reparameterized distance, and the volume are the usual up to a constant. 
\end{rem}
\begin{rem}
 In \cite{wuwu2024}, the case $\Ric_f^\infty \ge0$, $|f| \le k$ was studied.    Their volume ratio was defined by comparing  $r^{n+1 +4k}$, but in this case the volume grows at most like $r^{n+1}$, so their volume ratio would always be zero. We compare with the correct power $r^{n+1}$. 
\end{rem}
In the special case $K=0$, we get 
\begin{cor}  \label{Corollary K=0}
Let $(M^{n+m}, g, e^{-f}dvol)$ be a complete noncompact smooth metric measurement space,  $\Sigma^n$ a closed submanifold. Assume that $\operatorname{Ric}_{n,f}^{1} \ge 0$, then
\begin{equation}
\int_\Sigma e^{f(x)}  |\vec{H}_f|^n\,d\sigma(x) 
 \ge \frac{|\mathbb{S}^n|}{|\mathbb{S}^{n+m-1}|} \operatorname{RV}_{\mu,0}(\Sigma).  \label{K=0}
\end{equation}
In particular, when $f$ is bounded and $M^{n+m}$ has Euclidean volume growth, then $M^{n+m}$ has no compact submanifold with $\vec{H}_f =0$.
\end{cor}
When $m=1$, Corollary~\ref{Corollary K=0} recovers \cite[Theorem 4.10]{borghini2023comparison}, where, in addition, $H_f<0$ (in our notation) is assumed. 

When $f=0$, Corollary~\ref{Corollary K=0} gives \begin{equation}
\int_\Sigma   |H(x)|^n\,d\sigma(x) 
 \ge \theta \, |S^n|,   \label{K=0 f=0}
\end{equation}
where 
\begin{equation}
    \theta := \lims_{r \to \infty}\frac{\operatorname{vol}(B_x(r))}{|B^{n+m}| \,  r^{n+m}}.  \label{theta}
\end{equation} 
In particular, when $M^{n+m}$ is the Euclidean space, $\theta =1$, \eqref{K=0 f=0} recovers \eqref{Chen inequality}.

When $f=0$, $m=1$, $K=-1$,  Theorem~\ref{main theorem} is  \cite[(1.5)]{jin-yin}, when $\Sigma$ is the boundary of a bounded domain.

In Section~\ref{Section-Sol-Iso} we derive a Sobolev inequality  which improves \cite[Theorem 5.1]{Fujitani26} by replacing $\mathrm{Ric}_f^0 \geq 0$ by
$\mathrm{Ric}_f^1 \geq 0$, see Theorem~\ref{thm:Sobolev-inequa}. As a consequence, we get the following isoperimetric inequality. 
\begin{thm} \label{Iso inequa thm}
    Let $(M^n,g,e^{-f} d\vol)$ be a complete non-compact smooth metric measure space, $\Omega \subset M$ a bounded domain with smooth boundary.
    Assume that $\mathrm{Ric}_f^1 \geq 0$ and $f$ is bounded below. Then we have
    $$\operatorname{vol}_f(\partial \Omega) \geq n \,\theta_f^\frac{1}{n} e^{\frac{2 \inf_M f}{n}} (\mu_f(\Omega))^\frac{n-1}{n},$$
    where $d \mu_f = e^{-\frac{n+1}{n-1}f}dvol,$ \  $\theta_f := \limsup_{r \to \infty}\frac{\operatorname{vol}_f(B_x(r))}{r^n}$.
\end{thm}
This improves a result of Fujitani-Sakurai \cite{Fujitani26}, where $\mathrm{Ric}_f^0 \geq 0$ is assumed.

For $n\le 6$, this was proven by  Borghini-Fogagnolo in \cite{borghini2023comparison}.

    \section{Proof of Theorem~\ref{main theorem}}
    First, we use the normal exponential map to compute the volume of tubular neighborhoods of $\Sigma$ as in \cite{Heintze-Karcher}, see also \cite{Chahine} for the setup under $k$-Ricci curvature assumption. A key ingredient in our argument is the decomposition for the normal exponential map in \cite{Pan-Yi2026}. 

    % The proof of Theorem~\ref{main theorem} is based on an argument of \cite{Pan-Yi2026}.
    % A key ingredient in their work is a factorization formula for the Jacobian determinant of the normal exponential map, which expresses it as the product of an ambient exponential map term and a term encoding the geometry of the submanifold.
    % For reader's convenience, we recall their result.

    \subsection{Jacobian of the normal exponential map}
Following the general approach, to compute the volume of tube of submanifold, we compute the Jacobian of the normal exponential map.

For each $x \in \Sigma$,  $y \in S^\perp_x\Sigma$, the unit normal bundle of $\Sigma$, 
consider the normal exponential map
$$\Phi: S^\perp \Sigma \times [0,\infty) \to M \qquad \Phi(x,y,t) = \exp_{x}(ty) =:  \gamma_{x,y}(t).$$ 
Let $T_{\mathrm{cut}}(x,y)$ be the cut time of $\Sigma$ along $\gamma_{x,y}$. I.e. $T_{\mathrm{cut}}(x,y)$ is the largest $t$ such that the following holds.
\[ d(z, \Phi(x,y,t) ) \ge t
\]
for all $z \in \Sigma$. Then $\Phi$ maps $S^\perp \Sigma \times [0, T_{\mathrm{cut}}(x,y))$ onto $M$ up to a set of zero measurements. 
By the area formula,
$$\operatorname{vol}(T_r (\Sigma)) = \int_\Sigma \int_{S_x^\perp \Sigma} \int_0^{\min\{r, T_{\mathrm{cut}} (x,y)\}}|\operatorname{det} D\Phi_{(x,y,t)}|dt \, d\omega_{S_x^\perp \Sigma}(y) d{\operatorname{vol}}_{\Sigma}(x),$$
where $T_r (\Sigma)$ is a tubular neighborhood of $\Sigma$ of radius $r$ and $d\omega_{S_x^\perp \Sigma}$ is the standard measure on the unit sphere $S_x^\perp \Sigma \subset T^\perp_x \Sigma$.

For $0<t<T_{\mathrm{cut}}(x,y)$, we compute the Jacobian of $\Phi$ at $\gamma(t) = \gamma_{x,y}(t)$.
In \cite[Lemma 3.2]{Pan-Yi2026}, the authors showed that 
for the normal bundle exponential map 
$$\mathrm{Exp}^\perp(x,v)= \exp_x (v) \quad x \in \Sigma , \quad v \in T_x^\perp \Sigma,$$
one has factorization
$$|\operatorname{det} D \operatorname{Exp}^\perp_{(x,v)} | = |\operatorname{det} (D \exp_x )_{v} | |\operatorname{det} Q_v|,$$
where $$Q_v = \frac{1}{2} \operatorname{Hess} d^2_{\exp_x(v)}(x)|_{T_x \Sigma \times T_x \Sigma} - \langle \operatorname{II}_x, v\rangle.$$
Here $\exp_x:T_xM \to M$ denotes the exponential map of $M$ and $\operatorname{II}$ is the second fundamental form of $\Sigma$ defined by $\operatorname{II}_x(X,Y) = (\nabla_X Y)^\perp$.

Applying this formula with $v= ty$ and $|y| = 1$, we obtain 

\begin{equation}\label{eq:main factorization}
    |\mathrm{det} D\Phi_{(  x ,   y ,   t)}| = t^{m-1}|\operatorname{det} D \operatorname{Exp}^\perp(x,ty) | = t^{m-1}|\mathrm{det} (D\exp_{  x})_{  t  y}| | \mathrm{det} Q_t|,
\end{equation}
where 
%\exp_{  x} : T_{  x} M \to M$ is an exponential map %at $  x$ on $M$ and 
\begin{equation}\label{definition of Qt}
    Q_t :=\left.\frac{1}{2} \operatorname{Hess} d^2_{\gamma_{  x ,   y }(  t)} (   x)\right|_{T_{  x} \Sigma \times T_{  x} \Sigma} - \langle \mathrm{II}_x ,    ty\rangle .
\end{equation}

\subsection{Hessian comparison}
    In order to control $\operatorname{det} Q_t$, we first obtain a partial Hessian comparison for $\Ric_{l,f}^1$. 
    For this we get a Riccati inequality for the partial Hessian,
    %equation of a shape operator.
     then use the reparametrized distance $s$ as in \cite{wylie2016geometry} to naturally apply the assumption about $\Ric_{\ell, f}^1$. 

        \begin{prop}[Weighted partial Hessian comparison]\label{prop-weighted hessian}
        Let $(M^n,g,e^{-f}d\operatorname{vol})$ be a complete smooth metric measure space.
        Let
        $r_q(y):=d(q,y)$ be the distance function from a fixed point $q \in M$.
        Assume that
        \begin{equation}\label{Curvature-assumption}
        \mathrm{Ric}_{\ell,f}^1 \geq \ell K e^{-\frac{4f}{n-1}}
        \end{equation}
        for some $ K \in \mathbb{R}$ and some $1\leq \ell\leq n-1$.
        Let $p \in M \setminus (\mathrm{cut}(q)\cup \{q\})$, and let $\sigma:[0,T] \to M$ be the unique unit-speed minimizing geodesic from $q$ to $p$, where $T = r_q(p)$. Let $$V_l \subset (\nabla r_q)_{p}^\perp \subset T_p M$$ be an $\ell$-dimensional subspace, and let $\{e_1, \ldots, e_\ell\}$ be an orthonormal basis of $V_l$. Denote by $E_i(t)$ the parallel vector field along $\sigma$ satisfying $E_i(T) = e_i$.
        Then for every $t \in (0,T]$, with the additional restriction
        $s(t)< \pi / \sqrt{K}$ when $K>0$, 
        we have
        $$\sum_{i=1}^\ell (\operatorname{Hess} r_q)_{\sigma(t)}(E_i(t), E_i(t)) \leq \ell e^{-\frac{2f(t)}{n-1}} \frac{\mathrm{cs}_K(s(t))}{\mathrm{sn}_K(s(t))} + \frac{\ell}{n-1} f'(t),$$
        where $f(t) : = f(\sigma(t))$, and $s(t)$ is defined in \eqref{def:s}.
        \end{prop}
 When $l = n-1$, this estimate is equivalent to \cite[Lemma 4.2]{wylie2016geometry}. For completeness, we give a slightly different proof here. 
 
    \begin{proof}
        For $t \in (0,T]$, set 
        $$N(t):=\sigma'(t)=(\nabla r_q)_{\sigma(t)}.$$
        Thus $N(t)$ is the outward unit normal to the geodesic sphere centered at $q$.
        Define its shape operator by 
        $$\mathcal{S}_t(X): = (\nabla_X \nabla r_q)_{\sigma (t)}, \qquad X \in N(t)^\perp.$$
        It satisfies the Riccati equation (see \cite[Corollary 3.2.10]{petersen2016riemannian}):
        $$\nabla_{\sigma'(t)} \mathcal{S}_t = - R(\, \cdot \,, \sigma'(t)) \sigma'(t) - \mathcal{S}_t^2.$$
        
        Define
        $$V_l(t) := \mathrm{span}(E_1(t), \ldots, E_\ell(t))$$
        and
        $$q_\ell(t):= \sum_{i=1}^\ell(\operatorname{Hess} r_q)_{\sigma(t)} (E_i(t), E_i(t)) = \sum_{i=1}^\ell \langle \mathcal{S}_t(E_i) , E_i \rangle.$$
        Since each $E_i(t)$ is parallel along $\sigma$, the Riccati equation gives
        $$q_\ell' (t) = - \sum_{i=1}^\ell \langle R(E_i,\sigma'(t))\sigma'(t), E_i\rangle - \sum_{i=1}^\ell \langle \mathcal{S}_t^2(E_i), E_i\rangle$$
        By Cauchy-Schwarz inequality, we have
        \begin{equation}\label{Riccati-qn}
            q_\ell'(t) \leq - \mathrm{Ric}_\ell ({\sigma'(t)}, V_l(t))  - \frac{1}{\ell } q_\ell^2 (t).
        \end{equation}

        Since we do not have a lower bound on the $\ell$-Ric curvature, instead on $\Ric_{\ell,f}^1$, we adjust to the weighted ones.

Set $f(t) := f(\sigma(t))$
and define
$$
    q_{\ell,f}(t):=q_\ell(t)-\frac{\ell}{n-1}f'(t).
$$
Then from the Riccati inequality \eqref{Riccati-qn}, we have
\begin{align*}
q_{\ell,f}'(t)
&\leq
-\operatorname{Ric}_\ell(\sigma'(t),V_l(t))
-\frac{1}{\ell}q_\ell(t)^2
-\frac{\ell}{n-1}f''(t) \\
&=
-\operatorname{Ric}^{1}_{\ell,f}(\sigma'(t),V_l (t))
-\frac{1}{\ell}q_{\ell,f}(t)^2
-\frac{2}{n-1}q_{\ell,f}(t)f'(t).
\end{align*}

As the integrating factor of the linear part of the ODE is $e^{\frac{2f(t)}{n-1}}$, define
$$
    \lambda(t) : = e^{\frac{2f(t)}{n-1}} q_{\ell,f}(t).
$$
Then 
\begin{align*}
    \lambda'(t)
    &\leq
    e^{\frac{2f(t)}{n-1}}
    \left(
    -\operatorname{Ric}^{1}_{\ell,f}(\sigma'(t),V_l (t))
    -\frac{1}{\ell}q_{\ell,f}(t)^2
    \right)\\
    &  = e^{-\frac{2f(t)}{n-1}} \left(
    -\operatorname{Ric}^{1}_{\ell,f}(\sigma'(t),V_l (t)) e^{\frac{4f(t)}{n-1}}
    -\frac{1}{\ell}\lambda(t)^2
    \right).
\end{align*}
This motivates us to define the reparametrized distance function along $\sigma$ by
\begin{equation}\label{def:s}
    s(t) = \int_0^{t} e^{-\frac{2f(\sigma(\tau))}{n-1}}d \tau.
\end{equation}
Since its derivative is positive, the function $s$ is strictly increasing.
We denote its inverse by $t = t(s)$.

Set
$$
    \widetilde{\lambda}(s):=\lambda(t(s)).
$$
Then
$$
    \widetilde{\lambda}'(s) \leq -\operatorname{Ric}^{1}_{\ell,f}(\sigma'(t(s)),V_l (t(s))) e^{\frac{4f(t(s))}{n-1}}
    -\frac{1}{\ell}\lambda(t(s))^2.
$$
By the curvature assumption \eqref{Curvature-assumption},
we obtain
\begin{align}
\widetilde{\lambda}'(s)
&\leq
-\ell K
-\frac{1}{\ell }\widetilde{\lambda}(s)^2  \label{Riccati-tangential}
\end{align}
with the initial condition
\begin{equation*}
   \widetilde{\lambda}(s) = \frac{\ell}{s} +o(s)\quad \text{ as }s \to 0^+.
\end{equation*}

For comparison estimates, we use the standard \(\sn_K\)-notation:
\[
\sn_K''+K\sn_K=0,
\qquad
\sn_K(0)=0,
\qquad
\sn_K'(0)=1.
\]
That is,
\[
\sn_K(r)=
\begin{cases}
\dfrac{1}{\sqrt K}\sin(\sqrt K\,r), & K>0,\\[1ex]
r, & K=0,\\[1ex]
\dfrac{1}{\sqrt{|K|}}\sinh(\sqrt{|K|}\,r), & K<0.
\end{cases}
\]

Then
\[
\cs_K(r):=\sn_K'(r)=
\begin{cases}
\cos(\sqrt K\,r), & K>0,\\[1ex]
1, & K=0,\\[1ex]
\cosh(\sqrt{|K|}\,r), & K<0,
\end{cases}
\]
and \(\cs_K\) satisfies
\[
\cs_K''+K\cs_K=0,
\qquad
\cs_K(0)=1,
\qquad
\cs_K'(0)=0.
\]

By the singular Riccati comparison argument, we obtain

$$\widetilde{\lambda} (s)\leq \ell \frac{\operatorname{cs}_K(s)}{\operatorname{sn}_K(s)}$$
for $ 0 <s< S_{\mathrm{cut}}(q,\sigma'(0)) : = s(T_{\mathrm{cut}}(q,\sigma'(0)))$.
Thus, we obtain
$$\sum_{i=1}^\ell (\operatorname{Hess} r_q)_{\sigma(t(s))}(E_i(t(s)), E_i(t(s))) \leq \ell e^{-\frac{2f(t(s))}{n-1}} \frac{\mathrm{cs}_K(s)}{\mathrm{sn}_K(s)} + \frac{\ell}{n-1} f'(t(s)).$$
\end{proof} 
\begin{rem}[Rigidity in Proposition~\ref{prop-weighted hessian}]\label{rem:prop-1-1}
Fix $t_0 \in (0,T]$, and assume that 
$s(t_0)<\frac{\pi}{\sqrt K}$ when $K >0$.
Equality in Proposition~\ref{prop-weighted hessian} at $t = t_0$ holds if and only if for all $t \in (0,t_0]$,
\begin{equation}\label{eq:weighted-curvature-equality}
    \operatorname{Ric}_{\ell,f}^{1} (\sigma'(t),V_l(t)) = \ell K e^{-\frac{4f(t)}{n-1}}
\end{equation}
and
\begin{equation}\label{eq:shape-operator-equality}
    \mathcal{S}_t|_{V_l(t)} = \left(e^{-\frac{2f(\sigma(t))}{\mathrm{dim}M -1}} \frac{\operatorname{cs}_K(s(t))}{\operatorname{sn}_K(s(t))} + \frac{(f\circ\sigma)'(t)}{\mathrm{dim} M -1}\right) \operatorname{Id}_{V_l(t)}.
\end{equation}

In particular, if we set
$$R_{\sigma'(t)}(X) := R(X,\sigma'(t))\sigma'(t),$$
then the Riccati equation with \eqref{eq:shape-operator-equality} gives

\begin{equation}\label{eq:R-scalar}
    \left.R_{\sigma'(t)}\right|_{V_l(t)} = \left( K e^{-\frac{4f(\sigma(t))}{\mathrm{dim}M -1}} - \frac{f''(t)}{\mathrm{dim} M -1} - \frac{(f'(t))^2}{(\mathrm{dim} M -1)^2} \right) \operatorname{Id}_{V_\ell(t)}
\end{equation}
for every $t \in ( 0, t_0]$.
    
\end{rem}

\subsection{Determinant of $Q_t$ estimate}
The quantity $Q_t$ defined in \eqref{definition of Qt} is equivalently written as 
$$Q_t  = t \left(\operatorname{Hess} d_{\gamma_{  x ,   y}(  t)}\right)_x|_{T_x \Sigma \times T_x \Sigma} - t\langle \mathrm{II}_x, y \rangle .$$
Define
$$A_t : = \left(\operatorname{Hess} d_{\gamma_{  x ,   y}(  t)}\right)_x|_{T_x \Sigma \times T_x \Sigma} - \langle \mathrm{II}_x, y \rangle .$$
Then $Q_t = tA_t$ and $\operatorname{det} Q_t = t^n \operatorname{det} A_t$.

Fix $t_0$ and choose an orthonormal basis $\{e_i\}_{i=1}^n$ of $T_x\Sigma$ that diagonalizes $A_{t_0}$.
In our notation, \cite[Equation (3.24)]{Pan-Yi2026} gives the following differential inequality:
\begin{equation}\label{ineq:log-det-Qt}
    \left.\frac{d}{dt}\right|_{t = t_0} \log |\operatorname{det} Q_t|\leq \frac{n}{t_0} -n |\operatorname{det} A_{t_0}|^{-\frac{1}{n}} \left( \prod_{i=1}^n |D_{\gamma_{x,y}'(t_0)}J_i^{t_0}|\right)^\frac{2}{n},
\end{equation}
where for each $1 \leq i \leq n$, $J_{i}^{t_0}$ is the Jacobi field along $\gamma_{x,y}$ with boundary conditions 
$$J_i^{t_0}(0) = e_i, \quad J_i^{t_0} (t_0) = 0.$$

We first estimate $\operatorname{det} A_t$. By \cite[Lemma 3.4 (i)]{Pan-Yi2026}, $A_t$ is positive definite for $0<t<T_{\mathrm{cut}}(x,y)$, so we can apply the arithmetic--geometric mean inequality to its eigenvalues. Combining this with the partial Hessian comparison in Proposition~\ref{prop-weighted hessian}, we obtain the following estimate.
\begin{prop}\label{proposition At estimate}
For $0<t<T_{\mathrm{cut}}(x,y)$,
    $$|\operatorname{det} A_t| \leq \left(e^{-\frac{2f(x)}{n+m-1}} \frac{\mathrm{cs}_K(s(t))}{\mathrm{sn}_K(s(t))} - \frac{1}{n+m-1}\langle \nabla f (x), y \rangle - \langle \vec H(x) , y \rangle \right)^n  $$
\end{prop}
\begin{proof}
Fix $t_0$ and choose an orthonormal basis $\{e_i \}_{i=1}^n$ of $T_x \Sigma$.
Since $A_{t_0}$ is positive definite, the arithmetic--geometric mean inequality gives
\begin{align*}
    |\operatorname{det} A_{t_0} | 
    &\leq \left( \frac{\operatorname{tr}A_{t_0}}{n}\right)^n\\
    & = \left( \frac{\sum_{i=1}^n \left(\operatorname{Hess} d_{\gamma_{x,y}(t_0)}\right)_x (e_i , e_i) - \sum_{i=1}^n \langle \mathrm{II}_x (e_i, e_i) , y \rangle}{n}\right)^n\\
    & = \left( \frac{1}{n} \sum_{i=1}^n \left(\operatorname{Hess} d_{\gamma_{x,y}(t_0)}\right)_x (e_i , e_i) - \langle \vec H(x), y\rangle \right)^n\\
    & \leq\left(e^{-\frac{2f(x)}{n+m-1}} \frac{\mathrm{cs}_K(s(t_0))}{\mathrm{sn}_K(s(t_0))} - \frac{1}{n+m-1}\langle \nabla f(x), y \rangle - \langle \vec H(x) , y\rangle \right)^n,
\end{align*}
where we use Proposition~\ref{prop-weighted hessian} in the last inequality with $\ell = n$ and $\sigma (t) = \gamma_{x,y}(t_0 - t)$.
\end{proof}

\begin{rem}[Rigidity in Proposition~\ref{proposition At estimate}]\label{remark-rigidity-prop1-10}
Fix $t_0 \in (0,T_{\mathrm{cut}}(x,y)),$ and assume that 
$s(t_0)<\frac{\pi}{\sqrt K}$ when $K >0$. 
Consider the reversed geodesic 
$$\sigma (t) = \gamma_{x,y}(t_0 - t), \quad 0 \leq t \leq t_0.$$
Define
$$\bar s(t) : = \int_0^t e^{-\frac{2f(\sigma(\tau))}{n+m-1}}\,d\tau.$$
In particular, $\bar s (t_0) = s(t_0).$

Equality in Proposition~\ref{proposition At estimate} at $t = t_0$ holds if and only if \eqref{eq:shape-operator-equality} and \eqref{eq:R-scalar} hold with the reversed geodesic $\sigma(t)$ and reversed $\bar  s(t)$ and $V_n$ denotes the parallel transport of $T_x\Sigma$ along $\sigma$.

% the following conditions are satisfied. 
% For every $t \in (0, t_0]$,
% $$ \left.R_{\sigma'(t)}\right|_{V_n(t)} = \left( K e^{-\frac{4f(\sigma(t))}{n+m-1}} - \frac{(f\circ \sigma)''(t)}{n+m-1} - \frac{((f\circ \sigma)'(t))^2}{(n+m -1)^2} \right) \operatorname{Id}_{V_n(t)}$$
% and
% $$\mathcal{S}_t|_{V_n(t)} = \left(e^{-\frac{2f(\sigma(t))}{n+m -1}} \frac{\operatorname{cs}_K(\bar s(t))}{\operatorname{sn}_K(\bar s(t))} + \frac{(f\circ\sigma)'(t)}{n+m -1}\right) \operatorname{Id}_{V_n(t)}.$$
% Moreover,
% $$A_{t_0} = \left(e^{-\frac{2f(x)}{n+m-1}} \frac{\mathrm{cs}_K(s(t_0))}{\mathrm{sn}_K(s(t_0))} - \frac{1}{n+m-1}\langle \nabla f(x), y \rangle - \langle \vec H(x) , y\rangle\right)\operatorname{Id}_{T_x\Sigma}.$$
% Here $V_n(t)$ denotes the parallel transport of $T_x\Sigma$ along $\sigma$, and $\mathcal S_t$ is the shape operator associated with the distance function from $\gamma_{x,y}(t_0)$.
\end{rem}

It remains to estimate the second term on the right-hand side of \eqref{ineq:log-det-Qt}, namely
$$ \prod_{i=1}^n \left| D_{\gamma'_{x,y}(t_0)} J_i^{t_0} \right|.$$ Even though $J_i^{t_0}$ are Jacobi fields determined by the boundary condition, 
this term can be controlled by a traditional comparison for $\ell$-Jacobi fields determined by initial conditions. We therefore first derive  the following comparison. For the regular $\Ric_l$ lower bounds case (no weight), this is a classical result of Bishop, see e.g. \cite[Lemma 3.6]{Pan-Yi2026}.

\begin{prop}[Comparison of the determinant of $\ell$-Jacobi fields for $\operatorname{Ric}^{1}_{\ell,f}$]
    \label{lem:weighted-bishop-jacobi}
Let $(M^n,g,e^{-f}d\operatorname{vol})$ be a complete smooth metric measure space.
Fix $q \in M$ and $v \in T_qM$ with $|v| =1$. Consider the unit-speed minimizing geodesic
$$
    \sigma(t)=\exp_q(tv),\qquad 0 \leq t < T_{\mathrm{cut}}(q,v).
$$
Let $1\leq \ell\leq n-1$, and let
$w_1,\ldots,w_\ell\in T_qM$ be linearly independent vectors perpendicular to $v$.
Define the Jacobi fields
$$
    Y_i(t):=(D\exp_q)_{tv}(t w_i),\qquad 1\leq i\leq \ell .
$$
I.e. $Y_i$ is the Jacobi field with $Y_i(0) =0, \ Y'_i(0) = w_i$. 
Assume that
$$
    \operatorname{Ric}^{1}_{\ell,f}
    \geq
    \ell K e^{-\frac{4f}{n-1}}.
$$
Then, for every $0<t< T_{\mathrm{cut}}(q,v)$,
$$
\frac{|Y_1(t)\wedge\cdots\wedge Y_\ell(t)|}
{|w_1\wedge\cdots\wedge w_\ell|}
\leq
e^{\frac{\ell}{n-1}\left(f(\sigma(0))+f(\sigma(t))\right)}
\operatorname{sn}_K^\ell(s(t)) ,
$$
where $s(t)$ is defined in $\eqref{def:s}$.
\end{prop}

\begin{proof}
Let $L_t:v^\perp\to \sigma'(t)^\perp$ be a map defined by
$$
L_t (w)=(D\exp_q)_{tv}(t w).
$$
Then $Y_i(t)=L_t(w_i)$. Since $q$ has no conjugate points along $\sigma$,
$L_t$ is nonsingular for $t \in (0, T_{\mathrm{cut}}(q,v))$. Define
$$
W(t):=
\operatorname{span}\{Y_1(t),\ldots,Y_\ell(t)\}
$$
and
$$
J_\ell(t):=|Y_1(t)\wedge\cdots\wedge Y_\ell(t)|.
$$
In particular, $J_\ell (t) >0 $ for $t \in (0, T_{\mathrm{cut}}(q,v))$.

Let
$$
S_t:=L_t' \circ L_t^{-1},
$$
where $L_t'w$ denotes the covariant derivative of $L_t(w)$ along $\sigma$.
Equivalently, $S_t$ is the shape operator of the geodesic
sphere centered at $q$, so that
$$
\langle S_t(X),X\rangle
=
\left(\operatorname{Hess} r_q\right)_{\sigma(t)}(X,X),
\qquad X\in \sigma'(t)^\perp.
$$

For each $t >0$, choose an orthonormal basis $\{E_1(t),\ldots,E_\ell(t)\}$  of $W(t)$.
We claim that 
$$
\frac{d}{dt}\log J_\ell(t)
=
\sum_{i=1}^\ell \langle S_t(E_i(t)),E_i(t)\rangle .
$$
Indeed, if $G(t)=(\langle Y_i(t),Y_j(t)\rangle)_{i,j=1}^\ell$, then
$J_\ell(t)^2=\det G(t)$. Hence
\begin{align*}
\frac{d}{dt}\log J_\ell(t)
&=
\frac12\operatorname{tr}\bigl(G(t)^{-1}G'(t)\bigr).
\end{align*}
Using $Y_i'(t)=S_tY_i(t)$, this trace is precisely
$\operatorname{tr}_{W(t)}S_t$.

Applying Proposition~\ref{prop-weighted hessian} to the $\ell$-plane
$W(t)\subset \sigma'(t)^\perp$ at each $t$, we obtain
$$
\frac{d}{dt}\log J_\ell(t)
\le
\ell e^{-\frac{2f(t)}{n-1}}
\frac{\operatorname{cs}_K(s(t))}{\operatorname{sn}_K(s(t))}
+
\frac{\ell}{n-1}f'(t),
$$
where $f(t) : = f(\sigma(t))$ and $s(t)$ is defined in equation~\eqref{def:s}.
Since
$$
s'(t)=e^{-\frac{2f(t)}{n-1}},
$$
we have
$$
\frac{d}{dt}\log \operatorname{sn}_K(s(t))
=
e^{-\frac{2f(t)}{n-1}}
\frac{\operatorname{cs}_K(s(t))}{\operatorname{sn}_K(s(t))}.
$$
Therefore
$$
\frac{d}{dt}
\left(
\log J_\ell(t)
-
\ell\log \operatorname{sn}_K(s(t))
-
\frac{\ell}{n-1}f(t)
\right)
\le 0.
$$
Hence, for $0<\varepsilon<t$,
$$
\frac{J_\ell(t)}
{e^{\frac{\ell}{n-1}f(t)}\operatorname{sn}_K^\ell(s(t))}
\le
\frac{J_\ell(\varepsilon)}
{e^{\frac{\ell}{n-1}f(\varepsilon)}
\operatorname{sn}_K^\ell(s(\varepsilon))}.
$$

Since
$$
J_\ell(\varepsilon)
=
\varepsilon^\ell |w_1\wedge\cdots\wedge w_\ell|
+
O(\varepsilon^{\ell+2}) \quad \text{as }\varepsilon \to 0^+
$$
and
$$
s(\varepsilon)
=
e^{-\frac{2f(q)}{n-1}}\varepsilon
+
O(\varepsilon^2) \quad \text{as } \varepsilon \to 0^+,
$$
we have the limit
\begin{align*}
\lim_{\varepsilon\to0^+}
\frac{J_\ell(\varepsilon)}
{e^{\frac{\ell}{n-1}f(\varepsilon)}
\operatorname{sn}_K(s(\varepsilon))^\ell}
&=
e^{\frac{\ell}{n-1}f(q)}
|w_1\wedge\cdots\wedge w_\ell|.
\end{align*}
Consequently,
$$
J_\ell(t)
\le
e^{\frac{\ell}{n-1}\left(f(q)+f(\sigma(t))\right)}
\operatorname{sn}_K^\ell(s(t))
|w_1\wedge\cdots\wedge w_\ell|.
$$
Since
$$
J_\ell(t)=|Y_1(t)\wedge\cdots\wedge Y_\ell(t)|,
$$
we obtain
$$
\frac{|Y_1(t)\wedge\cdots\wedge Y_\ell(t)|}
{|w_1\wedge\cdots\wedge w_\ell|}
\leq
e^{\frac{\ell}{n-1}\left(f(q)+f(\sigma(t))\right)}
\operatorname{sn}_K(s(t))^\ell .
$$
This proves the desired estimate.
\end{proof}
\begin{rem}[Rigidity in Proposition~\ref{lem:weighted-bishop-jacobi}]
Fix $t_0 \in (0,T_{\mathrm{cut}}(q,v))$ and suppose $\operatorname{sn}_K(s(t)) >0$ for $ 0 < t \leq t_0$.
Equality in  Proposition~\ref{lem:weighted-bishop-jacobi} at $t = t_0$ holds if and only if for all $t \in (0,t_0]$, the equations \eqref{eq:shape-operator-equality} and \eqref{eq:R-scalar} hold with $V_\ell (t) = W(t)$.
Moreover, $W(t)$ is parallel transport of the $\ell$-plane $\mathrm{span}(w_1,\ldots, w_\ell)$ along $\sigma(t)$.

% \begin{equation*}
% \operatorname{Ric}_{\ell,f}^1(\sigma'(t), W(t)) = \ell K e^{-\frac{4f (\sigma(t))}{\mathrm{dim}M -1}}
% \end{equation*}
% and
% \begin{equation}\label{equality-hess}
%     S_t|_{W(t)} = \left(e^{-\frac{2f(\sigma(t))}{\mathrm{dim}M -1}} \frac{\operatorname{cs}_K(s(t))}{\operatorname{sn}_K(s(t))} + \frac{(f\circ\sigma)'(t)}{\mathrm{dim} M -1}\right) \operatorname{Id}_{W(t)}.
% \end{equation}
% Moreover, $W(t)$ is parallel transport of the $\ell$-plane $\mathrm{span}(w_1,\ldots, w_\ell)$ along $\sigma(t)$.
\end{rem}

We now apply the above proposition to $\sigma(t) = \gamma_{x,y}(t_0 - t)$ and 
$$Y_i (t) := (D\exp_{\gamma_{x,y}(t_0)})_{-t\gamma_{x,y}'(t_0)}\left(-t D_{\gamma_{x,y}'(t_0)}J_i^{t_0}\right)$$
for $1\leq i \leq n$.
Observe that
$Z_i(t):= J_i^{t_0}(t_0-t)$ is a Jacobi field along $\gamma(t)$ with 
$$Z_i ( 0) = 0, \quad  D_{\sigma'(0)}Z_i = -D_{\gamma_{x,y}'(t_0)}J_i^{t_0}.$$
By the uniqueness of Jacobi fields, 
$$Y_i(t) = Z_i(t).$$
Since $ Z_i(t_0) = J_i^{t_0}(0) = e_i$,
we have
$$|Y_1(t_0) \wedge \cdots \wedge Y_n(t_0)| = |e_1 \wedge \cdots \wedge e_n| = 1.$$
Then Proposition~\ref{lem:weighted-bishop-jacobi} with $\ell=n$ implies
$$\frac{1}{\left| D_{\gamma_{x,y}'(t_0)}J_1^{t_0} \wedge \cdots \wedge D_{\gamma_{x,y}'(t_0)}J_n^{t_0}\right|}\leq
e^{\frac{n}{n+m-1}\left(f(\gamma_{x,y}(0))+f(\gamma_{x,y}(t_0))\right)}
\operatorname{sn}_K(s(t_0))^n .$$
Thus,
$$ \prod_{i=1}^n \left| D_{\gamma'_{x,y}(t_0)} J_i^{t_0} \right|  \geq \left| D_{\gamma_{x,y}'(t_0)}J_1^{t_0} \wedge \cdots \wedge D_{\gamma_{x,y}'(t_0)}J_n^{t_0}\right| \geq e^{-\frac{n}{n+m-1}\left(f(t_0)+f(x)\right)}
\operatorname{sn}_K(s(t_0))^{-n},$$
where $f(t) : = f(\gamma_{x,y}(t)).$

By plugging the above inequality and Proposition~\ref{proposition At estimate} into inequality \eqref{ineq:log-det-Qt}, we have
\begin{align*}
    \left.\frac{d}{dt}\right|_{t=t_0} \log |\operatorname{det} Q_t| &\leq \frac{n}{t_0} - \frac{ne^{-\frac{2}{n+m-1} \left( f(t_0) +f(x)\right)} \operatorname{sn}_K^{-2}(s(t_0))}{e^{-\frac{2f(x)}{n+m-1}} \frac{\operatorname{cs}_K(s(t_0))}{\operatorname{sn}_K(s(t_0))} - \langle \vec H_f (x) , y\rangle }\\
    & = \left.\frac{d}{dt}\right|_{t=t_0} \log \left( t^n \left(e^{-\frac{2f(x)}{n+m-1}} \frac{\operatorname{cs}_K(s(t))}{\operatorname{sn}_K(s(t))} - \langle \vec H_f (x) , y\rangle  \right)^n\right),
\end{align*}
where $\vec H_f (x) = \vec H (x) + \frac{(\nabla f(x))^\perp}{n+m-1}$.
Thus,
\begin{equation}\label{def:theta-1}
    \theta_1(x,y,t): = \frac{|\operatorname{det}Q_t|}{t^n \left(e^{-\frac{2f(x)}{n+m-1}} \frac{\operatorname{cs}_K(s(t))}{\operatorname{sn}_K(s(t))} - \langle \vec H_f (x) , y\rangle  \right)^n}
\end{equation}
is monotone decreasing and $\lim_{t \to 0^+} \theta_1(x,y,t) = 1$.
Thus,
\begin{equation}\label{ineq:Qt estimate}
|\operatorname{det}Q_t| \leq t^n \left(e^{-\frac{2f(x)}{n+m-1}} \frac{\operatorname{cs}_K(s(t))}{\operatorname{sn}_K(s(t))} - \langle \vec H_f (x) , y\rangle  \right)^n.
\end{equation}

\subsection{Jacobian of $\Phi$} By \eqref{eq:main factorization}, to get the Jacobian estimate for $\Phi$, we also need to get an estimate for $|\mathrm{det} (D\exp_{  x})_{  t  y}|$. This is in \cite{wylie2016geometry}. Since it basically follows from the Laplacian comparison in Proposition~\ref{prop-weighted hessian}, for the convenience of the reader, we present a full proof here. 

The Jacobian $|\mathrm{det} (D\exp_{  x})_{  t  y}|$ represents the volume distortion of the exponential map from Euclidean tangent space $T_x M$ to $M$.
To express this in polar coordinates,
let $J(t,y)$ denote the volume density of a polar coordinate centered at $x$, so that $$d \mathrm{vol}_g = J(t,y) \,dt \,dy.$$
Since the Euclidean volume density on $T_x M$ in polar coordinates is $t^{n+m-1} dt dy$,
we obtain
$$|\operatorname{det} ( D \exp_x)_{ty}| = \frac{J(t,y)}{t^{n+m-1}}.$$

Since $$\partial_t (\log J(t,y)) = (\Delta r_x)_{\exp_x (ty)},$$ the estimate for $J(t,y)$ follows from the Laplacian comparison. 
By  Proposition~\ref{prop-weighted hessian} with $\ell = n+m-1$:
$$(\Delta r_x)_{\exp_x (ty)} \leq (n+m-1) e^{-\frac{2f(t)}{n+m-1}} \frac{\mathrm{cs}_K(s(t))}{\mathrm{sn}_K(s(t))} + f'(t),$$
where $f(t) := f(\exp_x(ty))$.
Then we have
$$\partial_t(\log (J(t,y) e^{-f(t)})) \leq \partial_t (\log (\operatorname{sn}_K^{n+m-1} (s(t))),$$
which implies that 
\begin{equation}\label{def:theta-2}
    \theta_2(x,y,t) : = \frac{J(y,t) e^{-f(t)}}{\operatorname{sn}_K^{n+m-1}(s(t))}
\end{equation}
is monotonically decreasing in $t$, which is {\cite[Lemma 4.3]{wylie2016geometry}}.

% we have
% $$\log J(t,y) \leq \log (\mathrm{sn}_K(s(t)))^{n+m-1} +\lim_{\epsilon \to 0^+} \log \left(\frac{J(\epsilon,y)}{(\mathrm{sn}_K(s(\epsilon)))^{n+m-1}}\right) + f(t)-f(x).$$
From the asymptotic expansions of $J(y,t)$ and $\mathrm{sn}_K(s(t))$, and the definition of $s(t)$, we have
$$\lim_{t \to 0^+}\theta_2 (x,y,t) = e^{f(x)}.$$
Thus,
$$J(t,y) \leq e^{f(t) + f(x)}\mathrm{sn}_K^{n+m-1}(s(t)).$$

    % \begin{prop}[{\cite[Lemma 4.3]{wylie2016geometry}}]
    % Let $(M^n,g,e^{-f}d\operatorname{vol})$ be a complete smooth metric measure space.
    %     If $\mathrm{Ric}_f^1 \geq (n-1) K e^{-\frac{4f}{n-1}}g$, then $\frac{e^{-f}J}{\mathrm{sn}_K(s_p)^{n-1}}$ is nonincreasing.
    % \end{prop}
    % By this monotonicity, we have
    % $$ J(t,y) \leq e^{f(x) +f ( \exp_x ( ty))} (\mathrm{sn}_K(s(t)))^{n+m-1}.$$

    Thus,
    $$|\operatorname{det} ( D \exp_x)_{ty}|  \leq \frac{e^{f(x) +f ( \exp_x ( ty))} \mathrm{sn}_K^{n+m-1}(s(t))}{t^{n+m-1}}.$$

   % \textcolor{red}{Is it better to write the above result as a proposition?}

    Now combining above with  inequality~\eqref{ineq:Qt estimate} and %the above inequality with the decomposition  
    \eqref{eq:main factorization},
    we obtain
\begin{prop}
    % \begin{align*}
    %     |\mathrm{det} D\Phi_{(  x ,   y ,   t)}|\leq
    %     t^n &\left(e^{-\frac{2f(t)}{n+m-1}} \frac{\mathrm{cs}_K(s(t))}{\mathrm{sn}_K(s(t))} - \frac{1}{n+m-1}\langle \nabla f (x) , y \rangle - \langle \vec H(x) , y \rangle \right)^n \\
    %     & \qquad \times \frac{e^{f(x) +f ( \exp_x ( ty))} (\mathrm{sn}_K(s(t)))^{n+m-1}}{t^{n+m-1}}
    % \end{align*}
    \begin{equation}
        |\mathrm{det} D\Phi_{(  x ,   y ,   t)}|\leq e^{f(x) +f (t)} \mathrm{sn}_K^{n+m-1} (s(t))  \left(e^{-\frac{2f(x)}{n+m-1}} \frac{\mathrm{cs}_K(s(t))}{\mathrm{sn}_K(s(t))} - \langle \vec{H}_f , y \rangle \right)^n, 
        \end{equation}
where $\vec{H}_f := \frac{1}{n+m-1} \nabla f(x) + \vec H(x)$. 
\end{prop}

    \subsection{Take limit and weighted volume ratio}
For $S>0$, we define the $s$-tube of $\Sigma$ by
$$
    C_S (\Sigma ) : = \{\exp_x(ty) \, : \, x \in \Sigma, \, y \in S_x^\perp \Sigma, \, 0 < s_{x,y }(t) < \min \{S, S_{\mathrm{cut}}(x,y)\}\},
$$
where $s_{x,y}(t) : = \int_0^t e^{-\frac{2f(\gamma_{x,y}(\tau))}{n+m-1}} d\tau$ and  $S_{\mathrm{cut}}(x,y) = s_{x,y}(T_{\mathrm{cut}}(x,y))$.
Consider the map $F : S^\perp \Sigma \times [0,S] \to M$ defined by
$$F(x,y,s) : = \exp_x(t_{x,y}(s)y),$$
where $t_{x,y}(s)$ is the inverse map of $s_{x,y}(t)$.
Then the map $F$ is a  diffeomorphism from the set
$$E_S : = \{(x,y,s)  \, : \, x \in \Sigma, \,y \in S_x^\perp \Sigma , \,0 < s< \min\{S, S_{\mathrm{cut}} (x,y)\} \}$$
onto $C_S(\Sigma)$.

For $(x,y,s ) \in E_S$, $F(x,y,s) = \Phi(x,y,t_{x,y}(s)).$
Since 
$$
    \frac{d t_{x,y}(s)}{ds} = e^{\frac{2f(t_{x,y}(s))}{n+m-1}},
$$
the chain rule gives
$$|\mathrm{det} DF_{(x,y,s)}| = |\mathrm{det} D\Phi_{(x,y,t_{x,y}(s))}|e^{\frac{2f(t_{x,y}(s))}{n+m-1}}.$$

Thus,
\begin{align*}
    |\mathrm{det} DF_{(x,y,s)}| &\leq e^{\left( 1 + \frac{2}{n+m-1}\right)f(t(s))}e^{f(x)}\left(e^{-\frac{2f(x)}{n+m-1}}\frac{\operatorname{cs}_K(s)}{\operatorname{sn}_K(s)} - \langle \vec H_f, y \rangle \right)^n \operatorname{sn}_K^{n+m-1}(s).
\end{align*}
That is, for any $(x,y) \in S^\perp \Sigma$ and $0< s< S_{\mathrm{cut}}(x,y)$, we have
\begin{eqnarray}
 \lefteqn{   
|\mathrm{det} DF_{(x,y,s)}| e^{-\frac{n+m+1}{n+m-1}f(\gamma_{x,y}(t_{x,y}(s)))} }\\ 
& & \leq 
e^{f(x)}\left(e^{-\frac{2f(x)}{n+m-1}}\frac{\operatorname{cs}_K(s)}{\operatorname{sn}_K(s)} -  \langle \vec H_f(x), y \rangle \right)^n \operatorname{sn}_K^{n+m-1}(s). \nonumber
\end{eqnarray}

Denote $$d\mu_f = e^{-\frac{n+m+1}{n+m-1}f}\, d \operatorname{vol}_g, $$
and define the weighted relative volume ratio, $\operatorname{RV}_{\mu,K}(\Sigma)$,  as
\begin{equation} \operatorname{RV}_{\mu,K}(\Sigma):= \limsup_{s \to \infty} \frac{\mu_f(C_S(\Sigma))}{\int_0^S \operatorname{sn}_K^{n+m-1}(s)ds }. 
%{ \vol (B(s) \subset \mathbb M_K^{n+m})} \cdot  \vol (\mathbb S^{n+m-1}), 
\label{Volume ratio}
\end{equation}
%where $\vol (B(s) \subset \mathbb M_K^{n+m})$ the volume of $s$-ball in the model space with constant curvature $K$.

Since $F:E_S \to C_S(\Sigma)$ is a diffeomorphism for any $S>0$, area formula implies
\begin{align}
        &\mu_f ( C_S(\Sigma)) \nonumber\\
        &\leq \int_\Sigma \int_{ S^\perp_x\Sigma}\int_0^{\min\{S, S_{\mathrm{cut}} (x,y)\}} e^{-\frac{n+m+1}{n+m-1} f(\gamma_{x,y}(t(s)))}|\mathrm{det} DF_{(x,y,s)}|\,ds \,dy\, d\sigma(x)\label{area-formula}\\
        & \leq \int_\Sigma \int_{ S^\perp_x\Sigma}\int_0^{S} e^{f(x)}
\left( e^{-\frac{2f(x)}{n+m-1}}
\frac{\operatorname{cs}_K(s)}{\operatorname{sn}_K(s)} - \langle \vec H_f (x), y \rangle 
\right)_+^n
\operatorname{sn}_K^{n+m-1}(s)\,ds \,dy\, d\sigma(x)\nonumber.
\end{align}
Dividing both sides by $\int_0^S \operatorname{sn}_K^{n+m-1}(s)ds$ and taking $\limsup$ as $S \to \infty$, we have
$$\mathrm{RV}_{\mu, K}(\Sigma) \leq
\int_\Sigma \int_{ S^\perp_x\Sigma}e^{f(x)} \left( e^{-\frac{2f(x)}{n+m-1}}\sqrt{-K} -\langle  \vec H_f (x), y \rangle \right)^n_+
 \,dy\, d\sigma(x)$$
because
\begin{align*}
\lim_{S\to\infty}&
\frac{
\displaystyle
\int_0^S
\left(e^{-\frac{2f(x)}{n+m-1}}
\frac{\operatorname{cs}_K(s)}{\operatorname{sn}_K(s)}- \langle  \vec H_f (x), y \rangle 
\right)_+^n
\operatorname{sn}_K^{n+m-1}(s)\,ds
}{
\displaystyle
\int_0^S \operatorname{sn}_K(s)^{n+m-1}\,ds
}\\
&\hspace{6cm}
=
\left( e^{-\frac{2f(x)}{n+m-1}}
\sqrt{-K} - \langle \vec H_f (x) , y \rangle 
\right)_+^n.
\end{align*}

% By substituting $a$ from \eqref{initial-a}, we have
% $$\mathrm{RV}_{\mu,K}(\Sigma) \leq 
% \int_\Sigma e^{f(x)}\!\int_{S_x^\perp \Sigma} \!\left( \! \sqrt{-K} e^{-\frac{2f(x)}{n+m-1}} + \!\left\langle \! -\vec H(x) - \frac{1}{n+m-1}\nabla f(x), y \right\rangle\right)_+^n\,dy d\sigma(x).$$
This finishes the proof of Theorem~\ref{main theorem}. 

If $K = 0$, we can further compute the integration 
by using the following (see Lemma 2.5 in \cite{brendle2026} for a similar proof)
\begin{equation}
     \int_{\mathbb S^{m-1}} (-\langle v, y \rangle )^n
_+ \, dy = |v|^n \frac{|\mathbb{S}^{n+m-1}|}{|\mathbb{S}^n|} \ \ \mbox{for any fixed} \ v \in \mathbb R^m,  \label{integral}
\end{equation}
Therefore we have 
$$\frac{|\mathbb{S}^n|}{|\mathbb{S}^{n+m-1}|}\mathrm{RV}_{\mu,K}(\Sigma) \leq 
\int_\Sigma e^{f(x)} \left| \vec H(x) +\frac{1}{n+m-1}(\nabla f(x))^\perp \right|^n d\sigma(x).$$
This is Corollary~\ref{Corollary K=0}. 

In particular, if $f =0$, then the left-hand side becomes $|\mathbb{S}^n|$ times  the usual asymptotic volume ratio $\theta$ defined in \eqref{theta}, which recovers the result by \cite{JK2025, Pan-Yi2026}.

\subsection{Equality case}
We now turn to the equality case in Theorem~\ref{main theorem} and show Proposition~\ref{prop:equality-Pulledback}.

Assume that $\Sigma$ is embedded and the following:
\begin{equation} 0< \mathrm{RV}_{\mu, K} (\Sigma) = \int_\Sigma e^{f(x)} \int_{S_x^\perp \Sigma} \left( \sqrt{-K}e^{-\frac{2f(x)}{n+m-1}} - \left\langle \vec H_f (x), y \right\rangle \right)_+^n \, dy d\sigma(x). \label{Eq:equality case}
\end{equation}

Let $$V : = \left\{(x,y) \in S^\perp \Sigma \, |\, \sqrt{-K}e^{-\frac{2f(x)}{n+m-1}} - \left\langle \vec H_f (x), y \right\rangle>0 \right\}.$$
For simplicity, let $a : = - \left\langle \vec H_f (x), y \right\rangle$.

%In order to match the the set $V$ with \textcolor{red}{???}
The set $V$ consists of the normal directions for which the limiting coefficient in the Jacobian estimate is strictly positive. On the other hand, before taking the limit, the Jacobian is expressed in terms of an $s$-dependent model factor. The following lemma relates the positivity of these two quantities. This allows us in Lemma~\ref{lem:equality-theta} to separate the directions with a positive asymptotic contribution from those whose normalized contribution vanishes.
\begin{lem}\label{lem:equality-setv}
    For $(x,y) \in S^\perp \Sigma$ and $K \leq 0$,
    $$\sqrt{ - K }e^{ - \frac{2f(x)}{n+m-1}} + a \geq 0 \, \iff \, e^{- \frac{2f(x)}{n+m-1}} \frac{\mathrm{cs}_K (s)}{\mathrm{sn}_K(s)} + a  \geq 0 \quad \text{for all } s >0.$$
\end{lem}
\begin{proof}
    If $a \geq 0$, then the above quantities are all non-negative.
    Suppose $a < 0$.

    ($\Rightarrow$) When $K =0$, it is trivially true. Below we assume $K \not= 0$.
    
    From the identity
    \begin{align*}
        & \mathrm{sn}_K(s) \left( e^{-\frac{2f(x)}{n+m-1}} \frac{\mathrm{cs}_K (s)}{\mathrm{sn}_K(s)} + a \right) = e^{-\frac{2f(x)}{n+m-1}} \mathrm{cs}_K(s) + a \mathrm{sn}_K (s) \\
        & = \frac{1}{2\sqrt{-K}} e^{\sqrt{-K}s} \left( e^{-\frac{2f(x)}{n+m-1}} \sqrt{-K} +a\right) + \frac{1}{2\sqrt{-K}}e^{-\sqrt{-K}s} \left( e^{-\frac{2f(x)}{n+m-1}} \sqrt{-K} -a\right).
    \end{align*}
   Since $a <0$, from above, we can see that
    $$e^{- \frac{2f(x)}{n+m-1}} \sqrt{-K} + a \geq 0 \implies e^{- \frac{2f(x)}{n+m-1}} \frac{\mathrm{cs}_K (s)}{\mathrm{sn}_K(s)} + a  \geq 0 \quad \text{ for all }s >0.$$

    ($\Leftarrow$) 
    % On the other hand, if
    % $e^{-\frac{2f(x)}{n+m-1}} \mathrm{cs}_K(s) + a \mathrm{sn}_K(s) \geq 0$ for all $s >0$, then,
    % $$\mathrm{sn}_K(s) \left( e^{-\frac{2f(x)}{n+m-1}} \frac{\mathrm{cs}_K (s)}{\mathrm{sn}_K(s)} + a \right) \geq 0 \quad \text{ for all }s >0.$$
    By taking limit, we have
    $$\lim_{s \to \infty} \left( e^{- \frac{2f(x)}{n+m-1}} \frac{\mathrm{cs}_K(s)}{\mathrm{sn}_K(s)} + a\right) \geq 0 .$$
    This gives the left hand side. 
    $$e^{- \frac{2f(x)}{n+m-1}}\sqrt{-K} + a \geq 0.$$
\end{proof}
From the definition of $\theta$ in equations \eqref{def:theta-1} and \eqref{def:theta-2}, we have
$$|\mathrm{det} Q_t| = t^n \left( e^{-\frac{2f(x)}{n+m-1}}\frac{\mathrm{cs}_K(s(t))}{\mathrm{sn}_K(s(t))} - \langle \vec H_f (x), y \rangle \right)_+^n \theta_1 (x,y,t)$$
and
$$|\mathrm{det} (D \exp_x )_{ty} | = \frac{\theta_2(x,y,t) e^{f(t)} sn_K^{n+m-1} (s(t))}{t^{n+m-1}}.$$
By plugging those identities into the factorization formula \eqref{eq:main factorization}, we get
$$|\mathrm{det} D\Phi_{(x,y,t)}| = \theta_1 \theta_2 (x,y,t) e^{f(t)} \mathrm{sn}_K^{n+m-1}(s(t)) \left( e^{-\frac{2f(x)}{n+m-1}} \frac{\mathrm{cs}_K(s(t))}{\mathrm{sn}_K(s(t))} - \langle \vec H_f (x), y \rangle \right)_+^n.$$

\begin{lem}\label{lem:equality-theta}
    For $(x,y) \in V$, we have $S_{\mathrm{cut}}(x,y) = + \infty$, $\theta_1(x,y,t) =1$ and $\theta_2(x,y,t) = e^{f(x)}$ for all $t >0$.
\end{lem}
\begin{proof}
    Let $$\theta_{x,y}(t(s)) := \begin{cases}
        \theta_1 ( x, y, t_{x,y}(s)) \theta_2 ( x,y, t_{x,y}(s)) &\mbox{ if } 0 < s< S_{\mathrm{cut}}(x,y)\\
        0 & \mbox{ if } s \geq S_{\mathrm{cut}}(x,y)
    \end{cases}$$
    where $t_{x,y}(s)$ is the inverse of $s_{x,y}(t)$.
    Since $\theta_1$ and $\theta_2$ are nonincreasing and
    $$\lim_{t\to0^+}\theta_1(x,y,t)=1,\qquad\lim_{t\to0^+}\theta_2(x,y,t)=e^{f(x)},
    $$
    we have
    $$
        0\leq \theta_{x,y}(s)\leq e^{f(x)}.
    $$
    
    By the area formula as in \eqref{area-formula}, we have 
    $$\mu_f ( C_S(\Sigma)) = \int_\Sigma \int_{S_x^\perp \Sigma} \int_0^S \theta(x,y,t(s)) \operatorname{sn}_K^{n+m-1}(s)\left( e^{-\frac{2f(x)}{n+m-1}} \frac{\operatorname{cs}_K(s)}{\operatorname{sn}_K(s)} - \langle \vec H_f (x), y \rangle \right)_+^n.$$
    Define
    $$G_{x,y}(S) : = \frac{\int_0^S \theta(x,y,t(s)) \operatorname{sn}_K^{n+m-1}(s)\left( e^{-\frac{2f(x)}{n+m-1}} \frac{\operatorname{cs}_K(s)}{\operatorname{sn}_K(s)} - \langle \vec H_f (x), y \rangle \right)_+^n\, ds}{\int_0^S \operatorname{sn}_K^{n+m-1}(s) ds}.$$
    Then $G_{x,y}(S)$ is uniformly bounded for $(x,y) \in S^\perp \Sigma$ and $S>1$.
    %(\textcolor{red}{Need to verify this? Seems OK as if we take the limit as S go to infinite it's bounded}). 
    Thus,
    $$\limsup_{S \to \infty} \frac{\mu_f (C_S (\Sigma))}{\int_0^S \operatorname{sn}_K^{n+m-1} (s) \, ds } \leq  \int_{\Sigma} \int_{S_x^\perp \Sigma}\limsup_{S \to \infty } G_{x,y}(S) \, dy \, d\sigma(x).$$

    If $(x,y)\notin V$, then $\sqrt{ - K }e^{ - \frac{2f(x)}{n+m-1}} + a\leq 0$.
    If $\sqrt{ - K }e^{ - \frac{2f(x)}{n+m-1}} + a<0$, then
    $$
        \bigl(e^{- \frac{2f(x)}{n+m-1}} \frac{\mathrm{cs}_K (s)}{\mathrm{sn}_K(s)} + a \bigr)_+=0
    $$
    for all sufficiently large $s$ by Lemma~\ref{lem:equality-setv}. If
    $\sqrt{ - K }e^{ - \frac{2f(x)}{n+m-1}} + a=0$, then
    $$
        \left(e^{- \frac{2f(x)}{n+m-1}} \frac{\mathrm{cs}_K (s)}{\mathrm{sn}_K(s)} + a \right)_+\longrightarrow 0
        \qquad\text{as }s\to\infty.
    $$
    Therefore, in either case,
    $$
        \lim_{S\to\infty}G_{x,y}(S)=0.
    $$

    % For $(x,y) \in S^\perp \Sigma \setminus V$, the integrand of the numerator becomes zero after some number by Lemma~\ref{lem:equality-setv}. Thus, for $(x,y) \in S^{\perp }\Sigma \setminus V$, we have
    % $$\lim_{S \to \infty} G_{x,y}(S) = 0.$$
    Suppose that $S_{\mathrm{cut}}(x,y)<+\infty$.
    Then $\theta_{x,y}(s)=0$ for all
    $s\geq S_{\mathrm{cut}}(x,y)$. Hence the numerator in
    $G_{x,y}(S)$ remains bounded as $S\to\infty$, whereas
    $$
        \int_0^S\sn_K(s)^{n+m-1}\,ds\longrightarrow+\infty.
    $$
    It follows that
    $$
        \lim_{S\to\infty}G_{x,y}(S)=0.
    $$

    Let $A : = \{ (x,y) \in V \, |\, S_{\mathrm{cut}(x,y)} = + \infty \}$.
    Then the above observations give
    $$\limsup_{S \to \infty} \frac{\mu_f (C_S (\Sigma))}{\int_0^S \operatorname{sn}_K^{n+m-1} (s) \, ds } \leq  \int_{A}\limsup_{S \to \infty } G_{x,y}(S) \, dy \, d\sigma(x).$$
    By the definition of $\mathrm{RV}_{\mu , K } (\Sigma) $ defined in \eqref{Volume ratio}, we have
    \begin{align*}
        \mathrm{RV}_{\mu , K } (\Sigma) & \leq \int_A \limsup_{S \to \infty} G_{x,y}(S) \, dy \, d \sigma (x)\\
        & = \int_A \limsup_{S \to \infty} \theta(x,y, t(s))  \left( e^{-\frac{2f(x)}{n+m-1}} \frac{\operatorname{cs}_K(s)}{\operatorname{sn}_K(s)} - \langle \vec H_f (x), y \rangle \right)_+^n\, dy \, d\sigma (x)\\
        &\leq \int_A e^{f(x)} \left( e^{-\frac{2f(x)}{n+m-1}}\sqrt{-K } - \langle \vec H_f (x), y \rangle  \right)_+^n dy \, d\sigma (x)\\
        & \leq \int_{V} e^{f(x)} \left( e^{-\frac{2f(x)}{n+m-1}}\sqrt{-K } - \langle \vec H_f (x), y \rangle  \right)_+^n dy \, d\sigma (x).
    \end{align*}
    By the equality assumption \eqref{Eq:equality case}, the last
    quantity is equal to $\operatorname{RV}_{\mu,K}(\Sigma)$.
    Therefore, equality holds throughout.
    This implies that
    \begin{enumerate}
        \item $V \setminus A$ has measure zero
        \item $\theta_1 (x,y,t) = 1$, $\theta_2 ( x,y, t) = e^{f(x)}$ a.e. $(x,y) \in A$ and every $t \geq 0$.
    \end{enumerate}
    The first one implies that $S_{\mathrm{cut}}(x,y) = \infty$ for almost everywhere in $V$. By the continuity of the normal cut-time function, $S_{\mathrm{cut}}(x,y) = \infty$ for all $(x,y) \in V$.
    Thus, for every $(x,y)\in V$ and every $t>0$,
    $$
        \theta_1(x,y,t)=1,
        \qquad
        \theta_2(x,y,t)=e^{f(x)}.
    $$

\end{proof}

By Lemma~\ref{lem:equality-theta}, we obtain
\begin{equation}\label{equality-dexp}
    |\mathrm{det} (D\exp_x)_{ty}| = \frac{e^{f(x) + f(t)}\operatorname{sn}_K^{n+m-1}(s(t))}{t^{n+m-1}}
\end{equation}
and
\begin{equation}\label{equality-qt}
    |\mathrm{det} Q_t| = t^n \left( e^{-\frac{2f(x)}{n+m-1}}\frac{\operatorname{cs}_K(s(t))}{\operatorname{sn}_K(s(t))} - \langle \vec H _f (x) , y \rangle \right)^n.
\end{equation}
% By plugging those identities into the factorization formula \eqref{eq:main factorization}, we get
% \begin{equation}
%     |\mathrm{det} D \Phi_{(x,y,t)}| = e^{\frac{-n+m-1}{n+m-1}f(x) + f(t)} \operatorname{sn}_K^{m-1}(s(t)) \left( \operatorname{cs}_K(s(t)) - e^{\frac{2f(x)}{n+m-1}}\langle \vec H_f (x), y \rangle \operatorname{sn}_K(s(t))\right)^n.
% \end{equation}

The equality in \eqref{equality-qt} implies the equality in Proposition~\ref{proposition At estimate} (See Remark~\ref{remark-rigidity-prop1-10}).
Thus, we have
\begin{equation}\label{Qt-equality}
    Q_t = t\left(e^{-\frac{2f(x)}{n+m-1}} \frac{\mathrm{cs}_K(s(t))}{\mathrm{sn}_K(s(t))} - \langle \vec H_f(x) , y\rangle\right)\operatorname{Id}_{T_x\Sigma}.
\end{equation}

On the other hand, the equality in \eqref{equality-dexp} implies the equality case in Proposition~\ref{prop-weighted hessian} (see Remark~\ref{rem:prop-1-1}). Thus,
$$\left.R_{\gamma_{x,y}'(t)}\right|_{(\gamma_{x,y}'(t))^\perp} = \left( K e^{-\frac{4f(\gamma_{x,y}(t))}{n+m -1}} - \frac{(f\circ\gamma_{x,y})''(t)}{n+m -1} - \frac{((f\circ \gamma_{x,y})'(t))^2}{(n+m -1)^2} \right) \operatorname{Id}_{(\gamma_{x,y}'(t))^\perp}$$
and
$$(\operatorname{Hess} r_x)_{\gamma_{x,y}'(t)}  = \left(e^{- \frac{2f(\gamma_{x,y}(t))}{n+m-1}} \frac{\operatorname{cs}_K(s(t))}{\operatorname{sn}_K(s(t))} + \frac{(f \circ\gamma_{x,y})'(t)}{n+m-1}\right)\operatorname{Id}_{(\gamma_{x,y}'(t))^\perp}.$$
This implies that for $z \in y^\perp \subset T_x M$
\begin{equation}\label{Et-equality}
    D (\exp_x)_{ty}(z) = \left( \frac{\operatorname{sn}_K(s(t))}{t} e^{\frac{f(x) + f(t)}{n+m-1}}\right)P_tz,
\end{equation}
where $P_tz$ is a parallel transport of $z$ along $\gamma_{x,y}(t)$.

\begin{prop}\label{prop:equality-Pulledback} Assume that $\Sigma$ is embedded and \eqref{Eq:equality case} holds. Then for every $(x,y,t)\in V\times(0,\infty)$, the pullback metric of
    $\Phi(x,y,t)=\exp_x(ty)$ is given by
    $$\Phi^*g = (a_t  b_t )^2 g_{\mathcal H} + a_t^2 g_{\mathcal V} + dt^2.$$
    Here $\mathcal H$ and $\mathcal V$ denote, respectively, the horizontal and vertical distributions on $S^\perp\Sigma$ induced by the normal connection, $a_t$ and $b_t$ are given in \eqref{eqn:a_t} and \eqref{def-b-t} respectively. Thus, at $(x,y)$,
    $$\mathcal H_{(x,y)}\simeq T_x\Sigma, \qquad \mathcal V_{(x,y)}\simeq T_x^\perp\Sigma\cap y^\perp$$
    and $g_{\mathcal H}$ and $g_{\mathcal V}$ denote the corresponding induced metrics.
\end{prop}
\begin{proof}
    Define 
    $$R: S^\perp \Sigma \times (0,\infty) \longrightarrow T^\perp \Sigma \setminus \{0\}, \quad R(x,y,t) = (x,ty).$$ 
    Then
    $$\Phi = \exp^\perp \circ R.$$

    Fix $(x,y,t)$ and write
    $$W=X^H+U^V+c\partial_t,$$
    where
    $X \in T_x\Sigma$ and $U\in T_x^\perp\Sigma\cap y^\perp.$
    Here $X^H$ and $U^V$ denote the horizontal and vertical lifts, respectively. We have
    $$dR_{(x,y,t)}(W) = X^H+(tU+cy)^V.$$

    Set
    $$E_t:=(d\exp_x)_{ty}:T_xM\longrightarrow T_{\gamma(t)}M, \qquad  \gamma(t):=\exp_x(ty)$$
    By \cite[Lemma 3.2]{Pan-Yi2026},  together with the vanishing of the mixed block, we obtain
    $$(d\exp^\perp)_{(x,ty)}(X^H+\xi^V) = E_t(Q_tX+\xi).$$
    Consequently, using \eqref{Qt-equality},
    \begin{align*}
        d\Phi_{(x,y,t)}(W)  =  E_t(Q_tX+tU+cy)
         = E_t(t b_t X+tU+cy),
    \end{align*}
    where 
    \begin{equation}\label{def-b-t}
        b_t  = e^{-\frac{2f(x)}{n+m-1}} \frac{\operatorname{cs}_K (s(t))}{\operatorname{sn}_K(s(t))} - \langle \vec H_f(x), y \rangle. 
    \end{equation}

    Let $P_t:T_xM\to T_{\gamma(t)}M$ denote parallel transport along $\gamma$. By \eqref{Et-equality},
    $$E_t(v)=\frac{a_t}{t}P_tv \quad\text{for }v\in y^\perp, \qquad E_t(y)=\gamma'(t).$$
    Since $X,U\in y^\perp$, it follows that
    $$d\Phi_{(x,y,t)}(W)= a_t b_t  P_tX + a_tP_tU + c\gamma'(t),$$
    where
    \begin{equation}
        a_t = e^{\frac{f(x) + f(t)}{n+m-1}} \operatorname{sn}_K(s(t)).  \label{eqn:a_t}
        \end{equation}

    Now let
    $$W'=(X')^H+(U')^V+c'\partial_t.$$    
    Then 
    \begin{align*}
        \Phi^* g (W,W')
         = a_t^2 b_t^2 g (X , X') + a_t^2  g( U, U')  + cc'.
    \end{align*}
    This proves
    $$\Phi^*g = (a_t  b_t )^2 g_{\mathcal H} + a_t^2 g_{\mathcal V} + dt^2.$$
    \end{proof}

\section{Sobolev and Isoperimetric Inequalities for $\Ric_f^1 \ge 0$}
\label{Section-Sol-Iso}
We first prove a Sobolev inequality which improves \cite[Theorem~5.1]{Fujitani26} by
weakening the curvature assumption from $\mathrm{Ric}_f^0 \geq 0$ to
$\mathrm{Ric}_f^1 \geq 0$. This will imply Theorem \ref{Iso inequa thm}. 

The main idea is to choose $e_1$ as the geodesic direction $\nabla u(x)$.
With this choice, $Q_{11}$ can be estimated directly, while the remaining
trace $\sum_{i=2}^{n}Q_{ii}(t)$ 
satisfies a Riccati-type inequality with the desired coefficient $\frac{1}{n-1}$.
The two estimates are combined using the AM--GM inequality,
following the argument in \cite[Page~10]{chong2025sobolev}.

We now prove the following Sobolev inequality, which is Theorem 5.1 in \cite{Fujitani26}, but with $\mathrm{Ric}_f^0 \geq 0$ replaced by
$\mathrm{Ric}_f^1 \geq 0$. Originally, this was proven in \cite{Brendle2023} for $\Ric \ge 0$ (i.e. $f =0$). 

\begin{thm} \label{thm:Sobolev-inequa}
    Let $(M^n, g, e^{-f}dvol)$ be a complete noncompact smooth metric measure space satisfying $\mathrm{Ric}_f^1 \geq 0$ and suppose that $f$ is bounded from below.
 Let
    $$\theta_f: = \limsup_{r \to \infty} \frac{\operatorname{vol}_f(B_x(r))}{r^n},$$
    where $x \in M$ is arbitrary.
    Then for any bounded smooth domain $\Omega \subset M$
    and any positive function $\phi \in C^\infty (\overline{\Omega})$, the following inequality holds.
    \begin{equation}
    \int_\Omega \left|\nabla \left(e^{\frac{2f}{n}}\phi\right)\right|\,d\operatorname{vol}_f +  \int_{\partial \Omega} e^{\frac{2f}{n}}\phi \,d\sigma_f \geq n \,\theta_f^\frac{1}{n} e^{\frac{2}{n}\inf_Mf}\left(\int_\Omega \phi^{\frac{n}{n-1}} \,d\operatorname{vol}_f \right)^{\frac{n-1}{n}}. \label{Sobolev-inequa}
    \end{equation}
\end{thm} 
\begin{proof} As in the proof of \cite[Theorem 1.1]{Brendle2023}, by multiplying the function $\phi$ by a suitable constant, 
we may assume that  
\begin{equation}
 \int_\Omega \left|\nabla \left(e^{\frac{2f}{n}}\phi\right)\right|\,d\operatorname{vol}_f+  \int_{\partial \Omega} e^{\frac{2f}{n}}\phi \,d\sigma_{f} = n \int_\Omega \phi^{\frac{n}{n-1}} \,d\operatorname{vol}_f.   \label{int-equality}
\end{equation}
Let $u$ be a solution to the following PDE.
$$\begin{cases}
    \mathrm{div} \left( e^{-\left(1-\frac{2}{n} \right)f } \phi \nabla u  \right) =  e^{-f} \left( n  \phi^{\frac{n}{n-1}} - \left|\nabla \left( e^{\frac{2f}{n}}\phi \right)\right|\right) &\mbox{ in }\Omega\\
    u_\nu =1 &\mbox{ on  }\partial \Omega
\end{cases}.$$
The compatibility condition for this Neumann problem follows from \eqref{int-equality}.
Thus the standard existence theory for uniformly elliptic Neumann problems gives a solution $u$, unique up to an additive constant. Since coefficients are continuous, elliptic regularity gives $u \in C^{2,\alpha}(\overline{\Omega})$.

Let  \begin{equation} U = \{x \in \Omega\,| \, 0< |\nabla u(x)| <1\}.  \label{U} \end{equation}
\begin{lem} \label{Delta_f est}
    For $x \in U$, we have $e^{\frac{2f}{n}}  \Delta_f u\leq n \phi^{\frac{1}{n-1}}$.
\end{lem}
\begin{proof}
Note that
\begin{align*}
    \mathrm{div} \left( e^{-\left(1-\frac{2}{n} \right)f } \phi \nabla u  \right) 
    & = e^{-\left(1-\frac{2}{n} \right)f } \phi \Delta u  + \langle \nabla (e^{-\left(1-\frac{2}{n}\right)f } \phi),\nabla u \rangle\\
    % & = e^{-\left(1-\frac{2}{n} \right)f } \phi \Delta u  + e^{-f} \langle \nabla (e^{\frac{2}{n}f}\phi ),\nabla u\rangle 
    % + e^{-f} e^{\frac{2f}{n}}\phi \langle \nabla (-f), \nabla u\rangle\\
    & = e^{-(1-\frac{2}{n})f}\phi \Delta_f u + e^{-f} \langle \nabla (e^{\frac{2}{n}f}\phi ),\nabla u\rangle.
\end{align*}
Thus, from the given PDE we have
$$e^{\frac{2f}{n}} \phi \Delta_f u = \left( n  \phi^{\frac{n}{n-1}} - \left|\nabla \left( e^{\frac{2f}{n}}\phi \right)\right|\right)-\langle \nabla (e^{\frac{2}{n}f}\phi ),\nabla u\rangle \leq n\phi^{\frac{n}{n-1}}$$
because $|\nabla u(x)| <1$.
\end{proof}
For each $r >0$, define a map
$$\Phi_r ( x) : = \exp_x (r \nabla u (x)): \Omega \to M$$
and a set
$$A_r : = \left\{x\in U \, \left|\,  ru(y) + \frac{1}{2} d(y, \exp_x (r \nabla u (x)))^2 \geq ru(x) + \frac{r^2}{2} |\nabla u(x)|^2\right.\forall y \in \Omega\right\}.$$

Recall the following lemma from \cite[Lemma 2.2]{chong2025sobolev}, see also \cite[Lemma 2.2]{Brendle2023}. 
\begin{lem}\label{lem:sobolev-image}
    The set $\{q \in M \setminus \Omega\, |\, d(x,q) <r \text{ for all } x \in \Omega \} $ is contained in $\Phi_r(A_r)$.
\end{lem}
% \begin{proof}[Proof of Lemma~\ref{lem:sobolev-image}]
%     Proof does not dependent on curvature assumption. It uses the contact set property (variation argument).
% \end{proof}

Fix $r>0$ and $x \in A_r$. Let $\{e_1, \ldots, e_n\}$ be an orthonormal basis for $T_x M$ with $e_1 = \frac{\nabla u}{|\nabla u|}.$ Note that $\nabla u(x) \not= 0$ by our definition of $U$ in \eqref{U}. 
For each $1 \leq i \leq n$, define the Jacobi fields $X_i$ along $\bar \gamma(t) = \exp_x ( t \nabla u(x))$ with initial conditions
$$X_i(0) = e_i, \quad X_i'(0) = \nabla_{e_i}\nabla u (x).$$
Let $E_i(t)$ be the parallel vector field along $\bar \gamma(t)$ with $E_i(0) = e_i$.

Define $n\times n $ matrices $P(t)$ and $R(t)$ by $$P_{ij}(t) = \langle X_i(t), E_j(t)\rangle, \quad R_{ij}(t) = R(E_i(t), \bar \gamma'(t), E_j(t), \bar \gamma'(t)).$$
Then the Jacobi equation gives
$$P'' + PR = 0, \quad P(0) = I_n, \quad P'(0) = \mathrm{Hess} \,u(x).$$
As shown in  \cite[Lemma 5.6]{brendle2026} implies that for $x \in A_r$ the Jacobi fields $X_i$ are independent and therefore $P(t)$ is invertible. 

Let $Q(t) = P^{-1}(t) P'(t)$. Then $Q$ is symmetric and satisfies the Riccati equation
$$Q'(t) + Q^2 (t) = -R(t), \quad Q(0) = \operatorname{Hess} u(x ).$$

We first use a one-dimensional Riccati comparison for the component of $Q(t)$ in the direction of $\nabla u(x)$, and then use a separate comparison for the trace of $Q(t)$ restricted to the orthogonal complement of this direction.
Combining these two estimates will give the desired bound for $\tr Q(t)$.

Since $R_{11}(t) = 0$ and $(Q_{11})^2 \leq (Q^2)_{11}$, we have
$$Q_{11}'(t) + Q_{11}^2 (t) \leq 0, \quad Q_{11}(0) = \operatorname{Hess} u(e_1,e_1).$$
By direct computation, we have
$$Q_{11}(t) \leq \frac{\operatorname{Hess} u (e_1,e_1)}{ (1+ \operatorname{Hess} u (e_1,e_1) t)_+}.$$

Define
$q(t) = \sum_{i=2}^n Q_{ii}(t)$.
By the Cauchy--Schwarz inequality,
$$q'(t) + \frac{1}{n-1}q(t)^2 \leq - \sum_{i=2}^n R(E_i(t), \bar \gamma'(t), E_i(t), \bar \gamma'(t)) = - \mathrm{Ric}(\bar \gamma'(t), \bar \gamma'(t))$$
with $q(0) = \Delta u(x) -\operatorname{Hess} u(e_1,e_1)$.

Define
\[
q_f(t):=q(t)-f'(t),
\]
where \(f(t):=f(\bar \gamma(t))\).
Then we have
\begin{align*}
q_f'(t)
%&=q'(t)-f''(t)\\
&\le -\Ric(\bar \gamma'(t),\bar \gamma'(t))
    -\frac{1}{n-1}(q_f(t)+f'(t))^2
    -f''(t)\\
   %  &= -\Ric_f^1(\bar \gamma'(t),\bar \gamma'(t))
   % +\frac{1}{n-1}(f'(t))^2
   % -\frac{1}{n-1}\bigl(q_f(t)^2+2q_f(t)f'(t)+(f'(t))^2\bigr)\\
&= -\frac{2}{n-1}f'(t) q_f(t)  -\frac{1}{n-1}q_f(t)^2 -\Ric_f^1(\bar \gamma'(t),\bar \gamma'(t)).
\end{align*}

Define \[
\lambda(t):=e^{\frac{2f(t)}{n-1}}q_f(t).
\]
Then
\begin{align*}
\lambda'(t)
&\leq e^{\frac{2f(t)}{n-1}}
   \left(
      -\Ric_f^1(\bar \gamma'(t),\bar \gamma'(t))
      -\frac{1}{n-1}q_f(t)^2
   \right) \\
   &\leq -\frac{1}{n-1}q_f(t)^2 e^{\frac{2f(t)}{n-1}}= -\frac{1}{n-1}e^{-\frac{2f(t)}{n-1}}\lambda^2(t).
\end{align*}
We recall that
\[
s_x(t)=\int_0^t e^{-\frac{2f(\bar \gamma_x(\tau))}{n-1}}\,d\tau.
\]
Then
\[
\frac{ds_x}{dt}=e^{-\frac{2f(\bar \gamma_x(t))}{n-1}},
\qquad
\frac{dt_x}{ds}=e^{\frac{2f(\bar\gamma_x(t(s)))}{n-1}},
\]
where \(t_x(s)\) is the inverse function of \(s_x(t)\).
Set
\[
\tilde\lambda(s):=\lambda(t(s)).
\]

Then
\begin{align*}
\tilde\lambda'(s)
=\lambda'(t(s))\frac{dt}{ds}
\le-\frac{1}{n-1}\tilde\lambda(s)^2
\end{align*}
with
\begin{equation} \tilde \lambda (0) = e^{\frac{2f(x)}{n-1}}(\Delta u (x) -\operatorname{Hess} u (e_1, e_1) -\langle \nabla f, \nabla u(x)\rangle).  \label{lambda-initial2}
\end{equation}
By comparison, we get
$$\tilde \lambda (s) \leq \frac{\tilde \lambda(0)}{\left(1+ \frac{\tilde\lambda(0)}{n-1}s\right)_+}.$$

Combining the above estimates, we have 
\begin{align*}
\frac{d}{ds}
\bigl(
\log\bigl(e^{-f(t(s))}\det P(t(s))\bigr)
\bigr)
&=
\frac{d}{ds}
\left(
-f(t(s))+\log(\det P(t(s)))
\right)\\
% &=
% -f'(t(s))\frac{dt}{ds}
% +\tr(Q(t(s)))\frac{dt}{ds}\\
&=
q_f(t(s))e^{\frac{2f(t(s))}{n-1}} + Q_{11}(t(s)) \frac{dt}{ds}\\
% &=
% \tilde\lambda(s) + Q_{11}(t(s)) \frac{dt}{ds}\\
&\leq\frac{\tilde \lambda(0)}{1+\frac{\tilde\lambda(0)}{n-1}s} + \frac{\operatorname{Hess} u (e_1,e_1)}{ 1+\operatorname{Hess} u (e_1,e_1) t(s)} \frac{dt}{ds}\\
&=
\frac{d}{ds}\log \left(\left( 1+ \frac{\tilde \lambda (0)}{n-1} s\right)^{n-1}_+\left( 1+\operatorname{Hess} u (e_1,e_1) t(s)\right)_+\right).
\end{align*}
That is,
\[
s\mapsto
\frac{e^{-f(t(s))}\det P(t(s))}{\left( 1+ \frac{\tilde \lambda (0)}{n-1} s\right)^{n-1}_+ \left( 1+ \operatorname{Hess} u (e_1,e_1) t(s)\right)_+}
\]
is monotone decreasing. 

Since
$$\lim_{s\to 0^+} \frac{e^{-f(t(s))}\det P(t(s))}{\left( 1+ \frac{\tilde \lambda (0)}{n-1} s\right)^{n-1}_+ \left( 1+ \operatorname{Hess} u (e_1,e_1) t(s)\right)_+} = e^{-f(x)},$$
we have 

\[
e^{-f(\bar \gamma_x(t_x(s)))}\det P(t_x(s))
\le e^{-f(x)}
\left( 1+ \frac{\tilde \lambda (0)}{n-1} s\right)^{n-1}_+ \left( 1+ \operatorname{Hess} u (e_1,e_1) t_x(s)\right)_+.
\]

By substituting $s = s_x(t)$, we have
$$e^{-f(t)}\det P(t) \leq e^{-f(x)}\left( 1+ \frac{\tilde \lambda (0)}{n-1} s_x(t)\right)^{n-1}_+ \left( 1+ \operatorname{Hess} u (e_1,e_1) t\right)_+.$$
Using AM-GM inequality \begin{equation}
    a^{n-1}b \leq \left(\frac{n-1}{n}a + \frac{1}{n}b\right)^{n} \label{AMGM-inequality}
\end{equation}
for $a ,b\geq 0$,
and \eqref{lambda-initial2},  we have 
{\small
\begin{align*}
    &e^{-f(t)}\det P(t) \\
    &\leq e^{-f(x)}\left( 1+ \frac{e^{\frac{2f(x)}{n-1}}(\Delta u (x) -u_{11}(x)\, -\langle \nabla f, \nabla u(x)\rangle)}{n-1} s_x(t)\right)^{n-1} \left( 1+ u_{11}(x)\, t\right)\\
    & = e^{-f(x)}ts_x(t)^{n-1}e^{2f(x)}\left( \frac{e^{-\frac{2f(x)}{n-1}}}{s_x(t)}+ \frac{(\Delta_f u (x) -u_{11}(x) )}{n-1} \right)^{n-1} \left( \frac{1}{t}+ u_{11}(x) \right)\\
    &\leq e^{-f(x)}ts_x(t)^{n-1}e^{2f(x)}
    \left(\frac{n-1}{n} \left(\frac{e^{-\frac{2f(x)}{n-1}}}{s_x(t)}+ \frac{(\Delta_f u (x) -u_{11}(x) )}{n-1} \right) + \frac{1}{n} \left(\frac{1}{t}+ u_{11}(x) \right)\right)^n\\
    % & = e^{-f(x)}ts_x(t)^{n-1}e^{2f(x)}
    % \left(\frac{n-1}{n} \frac{e^{-\frac{2f(x)}{n-1}}}{s_x(t)}+ \frac{(\Delta u (x) -\operatorname{Hess} u (e_1, e_1) -\langle \nabla f, \nabla u(x)\rangle)}{n} +  \frac{1}{nt}+ \frac{1}{n}\operatorname{Hess} u (e_1,e_1) \right)^n\\
    % & = e^{-f(x)}ts_x(t)^{n-1}e^{2f(x)}
    % \left(\frac{n-1}{n} \frac{e^{-\frac{2f(x)}{n-1}}}{s_x(t)}+ \frac{(\Delta u (x)  -\langle \nabla f, \nabla u(x)\rangle)}{n} +  \frac{1}{nt}\right)^n\\
    &= e^{-f(x)}ts_x(t)^{n-1}e^{2f(x)}
    \left(\frac{n-1}{n} \frac{e^{-\frac{2f(x)}{n-1}}}{s_x(t)}+ \frac{(\Delta_f u (x)  )}{n} +  \frac{1}{nt}\right)^n\\
    & = e^{-f(x)}ts_x(t)^{n-1} \left( \frac{n-1}{n} \frac{e^{\frac{-2f(x)}{n(n-1)}}}{s_x(t)} + \frac{(\Delta_f u(x)) e^{\frac{2f(x)}{n}}}{n} + \frac{e^{\frac{2f(x)}{n}}}{nt}\right)^n
\end{align*}
}
Then the area formula with Lemma~\ref{lem:sobolev-image} gives
\begin{align*}
    &\operatorname{vol}_f(\{q \in M \setminus \Omega\, |\, d(x,q) <r \text{ for all } x \in \Omega \})\\
    & \leq \int_{A_r} |\operatorname{det} D\Phi_r(x)|e^{-f(\Phi_r(x))}dvol_M(x)\\
    & = \int_{A_r} \operatorname{det} P(r) e^{-f(r)} dvol_M(x)\\
    % & \leq \int_{A_r} e^{-f(x)} rs_x(r)^{n-1} \left( \frac{n-1}{n} \frac{e^{\frac{-2f(x)}{n(n-1)}}}{s_x(r)} + \frac{(\Delta_f u(x)) e^{\frac{2f(x)}{n}}}{n} + \frac{e^{\frac{2f(x)}{n}}}{nr}\right)^n dvol_M(x)\\
    & \leq \int_{\Omega}  e^{-f(x)}rs_x(r)^{n-1} \left( \frac{n-1}{n} \frac{e^{\frac{-2f(x)}{n(n-1)}}}{s_x(r)} + \frac{(\Delta_f u(x)) e^{\frac{2f(x)}{n}}}{n} + \frac{e^{\frac{2f(x)}{n}}}{nr}\right)^n dvol_M(x)
\end{align*}

Note that
\begin{align*}
    &\limsup_{r\to \infty}\frac{\int_{\Omega} e^{-f(x)} rs_x(r)^{n-1} \left( \frac{n-1}{n} \frac{e^{\frac{-2f(x)}{n(n-1)}}}{s_x(r)} + \frac{(\Delta_f u(x)) e^{\frac{2f(x)}{n}}}{n} + \frac{e^{\frac{2f(x)}{n}}}{nr}\right)^n dvol_M(x)}{r^n}\\
    & \leq \int_\Omega e^{-f(x)}\limsup_{r \to \infty} \left( \frac{s_x(r)}{r}\right)^{n-1}\left(\frac{(\Delta_f u(x)) e^{\frac{2f(x)}{n}}}{n}\right)^n\, dvol_M(x)\\
    & \leq \int_\Omega e^{-f(x)} \left(e^{-\frac{2\inf_M f}{n-1}}\right)^{n-1}\left(\frac{(\Delta_f u(x)) e^{\frac{2f(x)}{n}}}{n}\right)^n\, dvol_M(x)\\
    & \leq e^{-2\inf_M f}\int_\Omega e^{-f(x)}\left(\frac{(\Delta_f u(x)) e^{\frac{2f(x)}{n}}}{n}\right)^n\, dvol_M(x)
\end{align*}
where we used the property $f$ is bounded below.

%Since $$\Delta_f u(x) e^{\frac{2f(x)}{n}} \leq n\phi^{\frac{1}{n-1}},$$
Using Lemma~\ref{Delta_f est} above gives 
$$\limsup_{r\to \infty}\frac{\operatorname{vol}_f(\{q \in M \setminus \Omega\, |\, d(x,q) <r \text{ for all } x \in \Omega \})}{r^n} \leq  e^{-2\inf_M f}\int_\Omega \phi^{\frac{n}{n-1}}dvol_f(x).$$
% Then for some $r_0>0$ we have
% $$\theta_f = \limsup_{r \to \infty} \frac{\operatorname{vol}_f(B_x(r-r_0))}{r^n} \leq e^{-2\inf_M f} \int_\Omega \phi^{\frac{n}{n-1}}dvol_f(x),$$
% %where $\theta := \limsup_{r \to \infty}\frac{\operatorname{vol}_f(B_x(r))}{r^n}$.
 That is,
$$\int_\Omega \phi^{\frac{n}{n-1}}d vol_f (x) \geq \theta_f \cdot e^{2 \inf_Mf}.$$
Substituting this into \eqref{int-equality}, we have
\begin{align*}
    \int_\Omega \left|\nabla \left(e^{\frac{2f}{n}}\phi\right)\right|dvol_f +  \int_{\partial \Omega} e^{\frac{2f}{n}}\phi \,d\sigma_{f}
    %& = n \int_\Omega \phi^{\frac{n}{n-1}} dvol_f\\
    & = n \left(\int_\Omega \phi^{\frac{n}{n-1}} dvol_f\right)^\frac{1}{n} \left(\int_\Omega \phi^{\frac{n}{n-1}} dvol_f \right)^{\frac{n-1}{n}}\\
    & \geq n \,\theta_f^\frac{1}{n} e^{\frac{2}{n}\inf_Mf}\left(\int_\Omega \phi^{\frac{n}{n-1}} dvol_f \right)^{\frac{n-1}{n}}.
\end{align*}

\end{proof}

By plugging in $\phi = e^{-\frac{2f}{n}}$ into \eqref{Sobolev-inequa}, we get Theorem~\ref{Iso inequa thm}.

\bibliographystyle{plain}
\bibliography{willmore}

\end{document}